\documentclass[10pt]{amsart}{}
\usepackage[english]{babel}
\usepackage{amsmath,amssymb,amsthm,amscd}
\usepackage{url}
\usepackage{color,hyperref}
\usepackage{stmaryrd}      

\theoremstyle{plain}
\newtheorem{theorem}{Theorem}[section]
\newtheorem{proposition}[theorem]{Proposition}
\newtheorem{lemma}[theorem]{Lemma}
\newtheorem{corollary}[theorem]{Corollary}

\theoremstyle{definition}
\newtheorem{definition}[theorem]{Definition}

\newtheorem{remark}[theorem]{Remark}
\newtheorem{conjecture}{Conjecture}

\newcommand{\Q}{\mathbb{Q}}
\newcommand{\N}{\mathbb{N}}
\newcommand{\Z}{\mathbb{Z}}
\newcommand{\R}{\mathbb{R}}

\newcommand{\sub}{\subseteq}
\newcommand{\set}[2]{\{#1\mid #2\}}
\newcommand{\conv}{\mathrm{conv}}

\newcommand{\gen}[1]{\langle #1\rangle}
\newcommand{\ri}{\mathrm{ri}}
\newcommand{\x}{\textit{\textbf{x}}}

\begin{document}
\title[Torsion factors of commutative monoid semirings]{Torsion factors of commutative monoid semirings}

\author[M.~Korbel\'a\v{r}]{Miroslav~Korbel\'a\v{r}}
\address{Department of Mathematics, Faculty of Electrical Engineering, Czech Technical University in Prague, Technick\'{a} 2, 166 27 Prague 6, Czech Republic}
\email{korbemir@fel.cvut.cz}

\thanks{The author acknowledges the support by the project CAAS CZ.02.1.01/0.0/0.0/16\_019/0000778 and by the bilateral Austrian Science Fund (FWF) project I 4579-N and Czech Science Foundation (GA\v{C}R) project 20-09869L ``The many facets of orthomodularity''}


\begin{abstract}
Let $P$ be a finitely generated commutative semiring. It was shown recently that if $P$ is a parasemifield (i.e. the multiplicative reduct of $P$ is a group) then $P$ cannot contain the positive rationals $\Q^+$ as its subsemiring. Equivalently, a commutative parasemifield $P$ finitely generated as a semiring is additively divisible if and only if $P$ is additively idempotent.    
 We generalize this result using weaker forms of these additive properties  to a broader class of commutative semirings in the following way. 
 
 Let $S$ be a semiring that is a factor of a monoid semiring $\N[\mathcal{C}]$ where $\mathcal{C}$ is a submonoid of a free commutative monoid of finite rank. Then the semiring $S$ is additively almost-divisible if and only if $S$ is  torsion. In particular, we show that if $S$ is a ring then $S$ cannot contain any non-finitely generated subring of $\Q$.
\end{abstract}

\maketitle

\section{Introduction}

In this paper we study properties of additive semigroups of commutative semirings. We generalize a result that was recently achieved for  commutative parasemifields \cite{parasemifields}  to a broader class of factors of certain type of monoid semirings. 
Our generalization will concern torsion properties of additive semigroups of commutative semirings in connection to a weaker types of divisibility investigated in \cite{torsion_div,div_gen}. 
As an immediate corollary we provide a characterization of certain class of commutative additively regular semirings. According to \cite{ilin,katsov},  semirings that are additively regular play an essential r\^{o}le in the theory of injective 
semimodules. 

By a \emph{semiring} $S(+,\cdot)$ we mean a non-empty set $S$ equipped with two operations, the addition $+$ and the multiplication $\cdot$ , such that $S(+)$ is a commutative semigroup, $S(\cdot)$ is a semigroup  and the multiplication distributes over the addition. 
A semiring $S$ is \emph{commutative} if $S(\cdot)$ is a commutative semigroup and  a \emph{parasemifield} if $S(\cdot)$ is a group. 
(For the detailed definitions of further properties of semirings and semigroups as idempotency, divisibility and others see Section \ref{preliminaries}.)

Every parasemifield contains either the semiring of positive rationals $\Q^{+}$ or it is additively idempotent.
Additively idempotent parasemifields are term-equivalent \cite{weinert} to significant algebraic objects - the lattice-ordered groups (that appear e.g. in functional analysis \cite{luxemburg} or in logic in connection with  MV-algebras \cite{mundici,mundici2}).
According to this term-equivalence, a lattice-ordered group is finitely
generated if and only if it is finitely generated as a parasemifield (where beside the operations $+$ and $\cdot$ also the inversion $a\mapsto a^{-1}$ is used). However, such a  case is not
equivalent to the property of being finitely generated as a semiring which is stronger.

Recently it was proved that every commutative parasemifield that is finitely generated as a semiring is additively idempotent \cite{parasemifields}. An essential property in the proof was the divisibility of the additive semigroup of the given parasemifield. By the result of \cite{parasemifields}, divisibility and idempotency are equivalent for additive semigroups of parasemifields that are finitely generated as semirings.  

Idempotent semigroups, on the other hand, can also be seen as a special case of torsion semigroups. For a finitely generated commutative semiring $S$ we can in this context ask for a weaker version of divisibility of $S(+)$ that will be equivalent to torsioness of $S(+)$. In \cite{torsion_div} such a suitable property, the almost-divisibility,  was studied and the following conjecture was suggested there.

\begin{conjecture}\label{conjecture}
Let $S$ be a finitely generated commutative semiring. Then the semigroup $S(+)$ is torsion if and only if $S(+)$ is almost-divisible. 
\end{conjecture}

In this paper we confirm this conjecture under an additional assumption that $S$ has a unity. Notice that such a result implies immediately Theorem 1.2 on parasemifields in \cite{parasemifields}, as the only parasemifields with torsion additive semigroup are those that are additively idempotent. 

We prove, in fact, a stronger result. Denote by $\mathfrak{S}$ the class consisting of all semirings that are factors of monoid semirings $\N[\mathcal{C}]$, where $\mathcal{C}$ is any submonoid of a  commutative monoid of a finite rank (see Section \ref{preliminaries} for the details). Clearly, every finitely generated semiring with a unity belongs to $\mathfrak{S}$, but many other semirings in $\mathfrak{S}$ are not of this kind.  In Theorem \ref{main-theorem}(i) we show that the equivalence  in Conjecture \ref{conjecture} can be extended to all semirings from $\mathfrak{S}$. Consequently, every semiring $S\in \mathfrak{S}$ cannot contain certain non-finitely generated subsemirings of $\Q^+$ (Theorem \ref{main-theorem}(iii)). In particular, we prove that if $S\in\mathfrak{S}$ is a \emph{ring}, then $S$ cannot contain \emph{any non-finitely} generated subring of the field of rationals $\Q$ (Corollary \ref{corollary_rings}).

Finally, we deal with those torsion elements in a semigroup that generate subgroups. A semiring $T$ where every element lies in some (generally infinite) subgroup of $T(+)$ is called additively regular and, due to the commutativity of $T(+)$, this semiring $T$ is also additively inverse. Regularity and inversion are classical properties studied in semigroups \cite{howie,lawson,petrich}. Within the \emph{additive} semigroups of semirings these properties were investigated, e.g., in   \cite{karvellas_2,pondelicek,sokratova,zeleznikow_1}. There is a remarkable characterization of additively regular semirings  that was achieved quite recently \cite{ilin}. The additively regular semirings (with a unity and a zero) are characterized as  precisely those semirings (with a unity and a zero)  where every semimodule has an injective envelope. 

In this paper we provide a different kind of a partial characterization. We show that a semiring $S\in\mathfrak{S}$ is both additively regular and torsion if and only if the semigroup $S(+)$ is strongly almost-divisible (Theorem \ref{main-theorem}(ii)). In this way we confirm Conjecture I in \cite{div_gen} for the case of semirings with a unity. 

All results in this paper concern commutative semirings. For \emph{non-commutative} cases even the problem whether a finitely generated \emph{ring} may contain the field $\Q$ seems to be still open.

Our approach will be rather geometrical. In comparison with \cite{parasemifields} our method will be substantially generalized and simplified.

\section{Preliminaries}\label{preliminaries}

By $\N$ we denote the set all of positive integers. In a semigroup $A(+)$ an element $a\in A$ will be called
\begin{itemize}
 \item \emph{idempotent} if $a+a=a$;
 \item \emph{torsion} if the semigroup $\set{\ell a}{\ell\in\N}$ is finite;
 \item \emph{regular} if there is $b\in A$ such that $a=a+b+a$;
 \item \emph{divisible} if for every $n\in\N$ there is $c\in A$ such that $a=nc$;
 \item \emph{strongly almost-divisible} if there is an infinite set of prime numbers $P\sub\N$ such that for every $p\in P$ there is $c\in A$ such that $a=pc$;
 \item \emph{almost-divisible} if there is $k\in\N$ such that the element $b=ka$ is strongly almost-divisible.
\end{itemize}

We also assign these properties to the semigroup $A$ itself, if every element of $A$ has the respective property (e.g., $A$ is \emph{regular}, if every element of $A$ is regular). Let us still mention that the semigroup $A$ is \emph{inverse} if it is regular and all idempotents of $A$ mutually commute. 

Similarly, we call a semiring $S(+,\cdot)$ to be \emph{additively regular} (\emph{additively idempotent}, etc.) if the additive semigroup $S(+)$ is regular (idempotent, etc.). 
According to a convention in ring theory, we make the only exception for the torsion property and call the semiring  $S$ to be \emph{torsion} if $S(+)$ is a torsion semigroup.
If the semiring $S$ has a (multiplicative) unity we will denote this element by $1_S$.

\begin{remark}
Let $A$  be a semigroup. An element $a\in A$ is both regular and torsion if and only if $a$ lies in a finite subgroup $G$ of $A$. Let $G$ be of order $n\in\N$. Then, obviously, $a$ is strongly almost-divisible (within $G$) with respect to the infinite set $P$ of all prime numbers $p\in\N$ such that $\gcd(p,n)=1$. With help of this observation, the following diagram of implications between the respective properties can be easily verified.
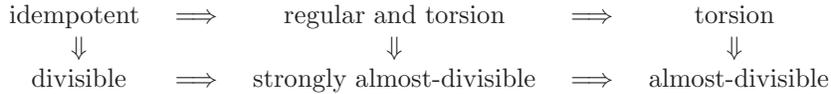
\begin{equation*}\label{diagram}
\begin{array}[c]{ccccc}
\text{idempotent}\ & \Longrightarrow &\  \text{regular and torsion}\  & \Longrightarrow & \text{torsion}  \\
\Downarrow & & \Downarrow & & \Downarrow \\
\text{divisible} & \Longrightarrow &\  \text{strongly almost-divisible}\  & \Longrightarrow &\ \text{almost-divisible} 
\end{array}
\end{equation*}

Our intention is to study conditions when the vertical implications turn into equivalences in  semirings. 
\end{remark}

Every vector space in this paper is assumed to be a \emph{real} vector space. 
For a subset $M\sub\R^n$ we denote by $\gen{M}$  the \emph{real vector subspace of $\R^{n}$ generated by $M$} and by $\conv(M)$ the \emph{convex hull of $M$}. Further, we denote by $\dim(M)$ the \emph{dimension of the convex hull} $\conv(M)$. The set $M$ is a  \emph{cone} if $M$ is a convex set such that for every non-negative real number $\lambda$ and every $u\in M$ we have $\lambda u\in M$. 
Notice that in the case when $M$ is a cone or a submonoid of $\R^n(+)$, the number $\dim(M)$ is the dimension of the vector space $\gen{M}$. 

When working with a convex set $C\sub\R^{n}$ we will use the notion of the \emph{relative interior of} $C$, denoted as $\ri(C)$ (see e.g. \cite{polytopes}). It is defined as the interior of $C$ with respect to the affine hull of $C$. 

By $\N_{0}^{n}$ we denote the set of all $n$-tuples of non-negative integers.
For a submonoid $\mathcal{C}$ of the free commutative monoid $\N_{0}^{n}(+)$ we define its  closure $\widetilde{\mathcal{C}}$ by setting $$\widetilde{\mathcal{C}}=\gen{\mathcal{C}}\cap \N_{0}^{n}\ .$$

In the sequel we will use the polynomial ring $\Z[x_1,\dots,x_n]$, i.e., we will work with polynomials with integer coefficients over a set of (commuting) variables $x_1,\dots,x_n$.  For $\alpha=(a_{1},\dots,a_{n})\in\N_{0}^{n}$ we put $\x^{\alpha}=x_{1}^{a_{1}}\dots x_{n}^{a_{n}} $ and for a submonoid $\mathcal{C}$ of $\N_{0}^{n}(+)$ we denote by $\Z[\mathcal{C}]$ the subring of $\Z[x_1,\dots,x_n]$ generated by the set of monomials $\set{\x^{\alpha}}{\alpha\in\mathcal{C}}$. 
The ring $\Z[\mathcal{C}]$ is a monoid ring based on the monoid $\mathcal{C}$ and the ring $\Z$. The unity of $\Z[\mathcal{C}]$ will be denoted simply by $1$. By $\N[\mathcal{C}]$ we denote the subsemiring of $\Z[\mathcal{C}]$ generated by the set $\set{\x^{\alpha}}{\alpha\in\mathcal{C}}$. To be precise, $\N[\mathcal{C}]=\set{\sum_{i=1}^m k_i \x^{\alpha_i}}{m\in\N, k_i\in\N, \alpha_i\in\mathcal{C}}\sub\Z[\mathcal{C}]$.

\section{Decomposition of saturated submonoids of $\N_{0}^n(+)$}\label{section 1}

A submonoid $\mathcal{C}$ of $\N_{0}^n(+)$ is called \emph{saturated} if for every $\alpha\in\N_{0}^n$ and every $k\in\N$ the condition $k\alpha\in\mathcal{C}$ implies that $\alpha\in\mathcal{C}$. This property can be interpreted in the way that there are no ``holes'' in the monoid $\mathcal{C}$. Obviously, the closure $\widetilde{\mathcal{C}}$ of $\mathcal{C}$ is always a saturated monoid.

For our considerations in further sections, we will first decompose a saturated monoid $\mathcal{C}\sub \N_{0}^{n}$ into saturated submonoids that will correspond to relatively open faces of the cone $K=\conv(\mathcal{C})\sub\R^n$ (Definition \ref{canonical_decomp} and Theorem \ref{monoid_decomp}). 
For the convenience of the reader we will provide more detailed proofs.

\begin{proposition}\label{cone_decomp}
For every cone $K$  in $\R^n$ there is a decomposition $\mathbf{K}=\set{A_i}{i\in I}$ of $K$ into disjoint union of relatively open convex subsets $A_i$ of $K$. 

Moreover, for every $A\in\mathbf{K}$ the set $A\cup\{0\}$ is a cone and for all $x\in A$ and $y\in K\setminus\gen{A}$, there is a relatively open convex set $B\in \mathbf{K}$ such that $\dim(B)>\dim(A)$ and the relative interior of the line segment $\conv(\{x,y\})$ lies in $B$.
\end{proposition}
\begin{proof}
It is easy to verify that the following construction provides the desired decomposition.
For every $x\in K$ there is a unique vector space $W_{x}$ of the maximal dimension such that $x$ is a relatively inner point of the convex set $W_{x}\cap K$. Now, set a relation $\sim$ on $K$ as  $x\sim y$  if and only if $W_x=W_y$. This relation is an equivalence and its partition sets are the desired relatively open faces of $K$. In particular, such a face $A$ is of the form $A=\ri(W_x \cap K)$, where $x\in A$.
\end{proof}

\begin{lemma}\label{fin_gen}
(i) Let $W\sub\R^n$ be a vector space with a basis that consists of vectors from $\Q^{n}$. Then the monoid $W\cap\N_{0}^n$ is finitely generated.

(ii) Let $C=\conv(a_1,\dots,a_k)$ be the convex hull of the points $a_1,\dots,a_k\in\R^{n}$. If $y\in C$ is a proper convex combination of $a_1,\dots,a_k$ (i.e., $y=\sum^{k}_{i=1}\lambda_i a_i$, where $0<\lambda_i<1$ for every $i=1,\dots, k$ and $\sum^{k}_{i=1}\lambda_i=1$) then $y\in\ri(C)$. 
\end{lemma}
\begin{proof}
(i) By Gordan's lemma \cite[Lemma 2.9]{polytopes}, $W\cap \Z^n$ is a finitely generated monoid. As the monoid $\N^{n}_{0}$ is also finitely generated, the intersection $(W\cap \Z^n)\cap \N_{0}^{n}=W\cap \N_{0}^{n}$ is, by \cite[Corollary 2.11]{polytopes}, again a finitely generated monoid.

(ii) The interior of a simplex in $\R^n$ consists of all proper convex combinations of its corner points. Thus our claim is true if $a_1,\dots,a_{k}$ are affinely independent. In a general case assume, without loss of generality, that for some $m\leq k$ the points $a_1,\dots,a_m$ are affinely independent and $a_{m+1},\dots,a_k\in\conv(a_1,\dots,a_m)$.  If $y\in C$ is a proper convex combination of $a_1,\dots,a_k$, then, by expressing the points $a_{m+1},\dots,a_k$ as convex combinations of $a_1,\dots,a_m$, it follows that $y$ is also a proper convex combination of $a_1,\dots,a_m$. Thus $y\in\ri(\conv(a_1,\dots,a_m))=\ri(C)$.
\end{proof}

\begin{lemma}\label{convex2}
  Let $\mathcal{C}\sub\N^n_0$ be a saturated monoid. Let $\mathbf{K}$ be the decomposition of the cone $K=\conv(\mathcal{C})\sub\R^n$ into relatively open faces as in Proposition $\ref{cone_decomp}$. For a relatively open face  $A\in\mathbf{K}$ of $K$ the following holds:
  \begin{enumerate}
\renewcommand{\theenumi}{(\roman{enumi})}
\renewcommand{\labelenumi}{\theenumi}
\item $(A\cap\mathcal{C})\cup\{0\}$ is a saturated submonoid of $\mathcal{C}$.
\item\label{oo} $A\cap\mathcal{C}=A\cap\N_{0}^n$.
\item\label{o} $\gen{A\cap\mathcal{C}}=\gen{A}$.
\end{enumerate} 
\end{lemma}
\begin{proof}
 (i) Obviously, the intersection of saturated monoids $(A\cup\{0\})\cap\mathcal{C}=(A\cap\mathcal{C})\cup\{0\}$ is a saturated monoid.

 (ii) By \cite[Proposition 2.22]{polytopes}, it holds that $\N_0^{n}\cap\conv(\mathcal{C})=\mathcal{C}$. Thus we have  $\mathcal{C}\cap A=\N_0^{n}\cap\conv(\mathcal{C})\cap A=\N_0^{n}\cap A$.

 (iii) The construction of the decomposition $\mathbf{K}$ in the proof of Proposition \ref{cone_decomp} implies that for the vector space $W=\gen{A}\sub\R^{n}$ there holds that $A=\ri(W\cap K)$. Further, by (i), the monoid $(A\cap\mathcal{C})\cup\{0\}$ is saturated. 
 
 To show the equality $\gen{A\cap\mathcal{C}}=\gen{A}$, it is therefore enough to prove that $\dim (A\cap\mathcal{C})=\dim (A)$. Let $k=\dim(A)$. For an element $y\in A=\ri(A)$ there are affinely independent points $y_1,\dots,y_{k+1}\in A\sub\R^n$ such that $y$ is their proper convex combination. Since $A\sub\conv(\mathcal{C})$, the point $y_1$ can be expressed by the least number $m\in\N$ of the points $z_1,\dots,z_m\in\mathcal{C}$. In particular, $y_1$ is a proper convex combination of these points. 
 
 Assume, for the contrary, that $\{z_1,\dots,z_m\}\not\sub W$. Then for the convex hull $B=\conv(z_1,\dots,z_m,y_2,\dots,y_{k+1})\sub K$ it holds that $k=\dim\big(\conv(y_1,\dots,y_{k+1})\big)\leq\dim(B)$ and $B\not\sub W$. Since $\dim W=k$, it follows that $k<\dim(B)$ and, by Lemma \ref{fin_gen}(ii), we obtain that $y\in\ri(B)$. This is a contradiction with the construction of open faces of $K$ (see the proof of Proposition \ref{cone_decomp}). 
 
 Hence it holds that $\{z_1,\dots,z_m\}\sub W$ and, by an analogous argument for the rest of $y_i$'s, there is a finite subset $F\sub\mathcal{C}$ such that $\conv(y_1,\dots,y_{k+1})\sub \conv(F)\sub W$. Thus, it holds that $\dim \big(\conv(F)\big)=k$. Now, we can easily find a set $F'\sub \R^n$ of $k+1$ points, that are affinely independent and all of them are proper rational convex combinations of points from $F$. In particular, by Lemma \ref{fin_gen}(ii), we obtain that $F'\sub\ri\big(\conv(F)\big)$. By the construction of $F'$, it follows that  $\ell\cdot F'\sub \mathcal{C}$ for some $\ell\in\N$. Further, since $\conv(F)\sub W\cap K$ and $\dim\big(\conv(F)\big)=\dim(A)=\dim(W)=\dim(W\cap K)$, it holds that $\ri\big(\conv(F)\big)\sub\ri(W\cap K)=A$.
 
 The set $M=\ell\cdot F'$ has therefore the following properties
\begin{itemize}
 \item $M\sub \mathcal{C}$,
 \item $M$ consists of $k+1$ points that are affinely independent, i.e., $\dim (M)=k$,
 \item $M=\ell\cdot F'\sub\ell\cdot \ri\big(\conv(F)\big) \sub \ell\cdot A\sub A$. 
\end{itemize}
Hence $M\sub A\cap\mathcal{C}$ and $\dim(M)=k$. Finally, we obtain that $\dim(A\cap\mathcal{C})=k=\dim(A)$ and, therefore, $\gen{A\cap\mathcal{C}}=\gen{A}=W$.
\end{proof}

\begin{definition}\label{canonical_decomp}
  Let $\mathcal{C}\sub\N^n_0$ be a saturated monoid and $\mathbf{K}$ be the decomposition of the cone $K=\conv(\mathcal{C})\sub\R^n$ into relatively open faces (see Proposition \ref{cone_decomp}). The system    $$\mathbf{D}(\mathcal{C})=\set{(A\cap\mathcal{C})\cup\{0\}}{A\in\mathbf{K}}$$ of saturated submonoids of $\mathcal{C}$ (see Lemma \ref{convex2}) will be called the \emph{canonical decomposition of the monoid} $\mathcal{C}$.
\end{definition}

\begin{remark}
A subsemigroup $B$ of a commutative semigroup $S$ is called a \emph{face} of $S$ if the conditions $a,b\in S$ and $ab\in B$ imply that $a,b\in B$ \cite[Chapter II, Section 6]{grillet}. 
Let us note that a monoid from $\mathbf{D}(\mathcal{C})$ does not need to be a face (submonoid) as it may miss some of its possible edge rays.
\end{remark}

\begin{proposition}\label{monoid_decomp}
  Let $\mathcal{C}\sub\N^n_0$ be a saturated monoid and $\mathbf{D}(\mathcal{C})$ be its canonical  decomposition. Then:
\begin{enumerate}
\renewcommand{\theenumi}{(\roman{enumi})}
\renewcommand{\labelenumi}{\theenumi}
\item $\mathcal{C}=\bigcup_{\mathcal D\in\mathbf{D}(\mathcal{C})} \mathcal D$,
\item $\mathcal D'\cap\mathcal D''=\{0\}$ for all $\mathcal{D}',\mathcal{D}''\in\mathbf{D}(\mathcal{C})$ such that $\mathcal D'\neq \mathcal D''$.
\end{enumerate} 
Further, for each $\mathcal{D}\in\mathbf{D}(\mathcal{C})$ the following holds:
\begin{enumerate}
\renewcommand{\theenumi}{(\roman{enumi})}
\renewcommand{\labelenumi}{\theenumi}
 \item\label{i} The monoid $\widetilde{\mathcal{D}}=\gen{\mathcal{D}}\cap\N_{0}^{n}$ is finitely generated.
 \item\label{ii} If $\dim(\mathcal{D})=\dim(\mathcal{C})$, then $\mathcal{C}\sub\widetilde{\mathcal{D}}$.
 \item\label{iii}   For all $\alpha\in \mathcal{D}\setminus\{0\}$ and $\beta\in \mathcal{C}\setminus \widetilde{\mathcal{D}}$ there is  $\mathcal{E}\in\mathbf{D}(\mathcal{C})$ such that $\dim(\mathcal{E})>\dim(\mathcal{D})$  and $\alpha+\beta\in \mathcal{E}$. 
 \item\label{iv}  For all $\alpha\in \mathcal{D}\setminus\{0\}$ and $\gamma\in \widetilde{\mathcal{D}}$ there is $k\in\N$ such that  $k\alpha+\gamma\in \mathcal{D}$.
\end{enumerate} 
\end{proposition}
\begin{proof}
For every $\mathcal{D}\in \mathbf{D}(\mathcal{C})$ there is a relatively open face $A\in\mathbf{K}$ of the cone $K=\conv(\mathcal{C})\sub\R^n$ such that $\mathcal{D}=(A\cap\mathcal{C})\cup\{0\}$.

 (i) Clearly, the vector space $\gen{\mathcal{D}}$ has a basis that consists of vectors from $\Q^{n}$. Hence, by Lemma \ref{fin_gen}(i), the monoid  $\widetilde{\mathcal{D}}$ is finitely generated.
 
 (ii) Since $\dim(\mathcal{D})=\dim(\mathcal{C})$ and $\mathcal{D}\sub\mathcal{C}$, we have $\gen{\mathcal{D}}=\gen{\mathcal{C}}$. Hence $\mathcal{C}\sub\gen{\mathcal{C}}\cap\N_{0}^n=\gen{\mathcal{D}}\cap\N_{0}^n=\widetilde{\mathcal{D}}$.

 (iii) Let $\alpha\in \mathcal{D}\setminus\{0\}$ and $\beta\in \mathcal{C}\setminus \widetilde{\mathcal{D}}$. By Lemma \ref{convex2}\ref{o}, we see that $\widetilde{\mathcal{D}}=\N_0^{n}\cap\gen{A}$ and $\beta\in K\setminus\gen{A}$. The rest follows immediately from Proposition \ref{cone_decomp} and from the fact that $\mathcal{C}$ is saturated.
 
 (iv) By Lemma \ref{convex2}\ref{o}, it holds that $\gen{A}=\gen{\mathcal{D}}$. The vector $\alpha\in \mathcal{D}\setminus\{0\}\sub A \cap \mathcal{C}\sub A=\ri(A)$  is an inner point of the convex set $A$. Hence every non-zero vector $u\in\gen{A}=\gen{\mathcal{D}}$ with a small enough angle between $u$ and $\alpha$ has to be contained in $A$. In particular, there is $\varepsilon>0$ such that for every $u\in\gen{\mathcal{D}}\setminus\{0\}$ the inequality  $\frac{|\alpha\cdot u|}{\|\alpha\|\cdot\|u\|}>1-\varepsilon$ implies that $u\in A$. 
 
 Further, for $\gamma\in\widetilde{D}=\gen{D}\cap\N_{0}^n$ we have
 $$\lim\limits_{k\to\infty}\frac{|\alpha\cdot (k\alpha+\gamma)|}{\|\alpha\|\cdot\|k\alpha+\gamma\|}=
 \lim\limits_{k\to\infty}\frac{|\alpha\cdot\alpha+\frac{\alpha\cdot\gamma}{k}|}{\|\alpha\|\cdot\|\alpha+\frac{\gamma}{k}\|}=1$$
  
 Therefore, by Lemma \ref{convex2}\ref{oo}, there is $k_0\in\N$ such that $k_0\alpha+\gamma\in A\cap\N_0^{n}=A\cap\mathcal{C}\sub\mathcal{D}$.
\end{proof}

\section{Additively almost-divisible factors of the monoid ring $\Z[\mathcal{C}]$  are torsion}\label{section 2}

In this section we prove results for the case of commutative rings. First let us state a basic result.

\begin{theorem}\cite[2.5]{torsion_div}\label{divisible}
Let $R$ be a finitely generated commutative ring. Then an element $a\in R$ is additively almost-divisible if and only if $a$ is  torsion.
\end{theorem}

For monoids $\mathcal{C},\mathcal{D}\sub\N^{n}_{0}(+)$ we define their sum as $\mathcal{C}+\mathcal{D}=\set{\alpha+\beta}{\alpha
\in\mathcal{C},\ \beta\in\mathcal{D}}$.

\begin{lemma}\label{decomp_lemma}
  Let $\mathcal{C}\sub\N_{0}^n$ be a saturated monoid  and  $\mathbf{D}(\mathcal{C})$ its canonical decomposition.
  Then for all $\mathcal{D}\in\mathbf{D}(\mathcal{C})$,  $\alpha\in\mathcal{D}\setminus\{0\}$ and  $f\in \Z[\mathcal{C}+\widetilde{\mathcal{D}}]$ there is $k\in\N$ such that $\x^{k\alpha}f\in \Z[\mathcal{C}]$.
\end{lemma}
\begin{proof}
 The set $\mathcal{C}+\widetilde{\mathcal{D}}=\set{\beta+\gamma}{\beta\in\mathcal{C}, \gamma\in\widetilde{\mathcal{D}}}$ is a submonoid of $\N^{n}_{0}$. Hence for $f\in \Z[\mathcal{C}+\widetilde{\mathcal{D}}]$ there are $\beta_1,\dots,\beta_m\in\mathcal{C}$, $\gamma_1,\dots,\gamma_m\in\widetilde{\mathcal{D}}$  and $\ell_1,\dots,\ell_m\in\Z$ such that $$f=\sum_{i=1}^{m} \ell_{i}\x^{\gamma_i+\beta_i}.$$
 
By Proposition \ref{monoid_decomp}\ref{iv}, there is $k\in\N$ such that  $k\alpha+\gamma_i\in \mathcal{D}\sub\mathcal{C}$ for every $i=1,\dots,m$.
Therefore we have  that  $$\x^{k\alpha}f=\sum_{i=1}^{m} \ell_{i}(\x^{k\alpha+\gamma_i})\x^{\beta_i}\in \Z[\mathcal{C}].$$
\end{proof}

For a ring $R$, an ideal $I$ of $R$ and an element $a\in R$ we will denote by $[a]_{I}$ the element of the factor-ring $R/I$ that corresponds to $a$.

\begin{lemma}\label{nil-radical}
  Let $\mathcal{C}\sub\N_{0}^n$ be a saturated monoid  and  $\mathbf{D}(\mathcal{C})$ be its canonical decomposition. Let $\mathcal{D}\in\mathbf{D}(\mathcal{C})$ and let $J$ be an ideal in the ring 
  $\Z[\mathcal{C}+\widetilde{\mathcal{D}}]$.
  
  Further, let 
     \begin{enumerate}
  \item[$\bullet$] $N$ be an ideal in the ring  $S=\Z[\mathcal{C}+\widetilde{\mathcal{D}}]/J$ such that 
   for every  $\mathcal{E}\in \mathbf{D}(\mathcal{C})$ with $\dim(\mathcal{E})>\dim(\mathcal{D})$  and every $\delta\in \mathcal{E}\setminus\{0\}$ it holds that $[\x^{\delta}]_{J}\in N$,
  \item[$\bullet$] $\pi:S\to S/N$ be the natural ring epimorphism and
  \item[$\bullet$] $T$ be the subring of $S/N$ generated by the set $\set{\pi\big([\x^{\beta}]_{J}\big)}{\beta\in \widetilde{\mathcal{D}}}$.
 \end{enumerate}
  Then
    \begin{enumerate}
  \renewcommand{\theenumi}{(\roman{enumi})}
\renewcommand{\labelenumi}{\theenumi}
   \item\label{b1}
  the ring $T$ is finitely generated,
  \item\label{b2} for every $f\in \Z[\mathcal{C}]$ and  $\alpha\in\mathcal{D}\setminus\{0\}$ we have $\pi\big([\x^{\alpha}f]_{J}\big)\in T$.
    \end{enumerate}
\end{lemma}
\begin{proof}
(i) By Proposition \ref{monoid_decomp}\ref{i}, the monoid $\widetilde{\mathcal{D}}$ is finitely generated. Therefore the ring $T$ is finitely generated, too.

(ii) Obviously, it is enough to prove the statement for the polynomials of the form $f=\x^{\gamma}$, where $\gamma\in\mathcal{C}$. If $\gamma\in \mathcal{C}\cap\widetilde{\mathcal{D}}$ then we have $\alpha+\gamma\in \widetilde{\mathcal{D}}$ and therefore we obtain that  $\pi\big([\x^{\alpha}f]_{J}\big)=\pi\big([\x^{\alpha+\gamma}]_{J}\big)\in T$.

On the other hand, if $\gamma\in\mathcal{C}\setminus\widetilde{\mathcal{D}}$ then, by Proposition \ref{monoid_decomp}\ref{iii}, there is $\mathcal{E}\in\mathbf{D}(\mathcal{C})$ such that $\dim(\mathcal{E})>\dim(\mathcal{D})$  and $\alpha+\gamma\in \mathcal{E}$. Hence, by the assumption on the ideal $N$, it holds that $[\x^{\alpha+\gamma}]_{J}\in N$. Therefore we obtain that  $\pi\big([\x^{\alpha}f]_{J}\big)=\pi\big([\x^{\alpha+\gamma}]_{J}\big)=0\in T$ and this concludes the proof.
\end{proof}

For a commutative ring $R$ put $$\mathcal{N}_{\mathcal{T}}(R)=\set{a\in R}{(\exists\ \ell,q\in\N)\ q\cdot a^{\ell}=0}\ .$$
It is easy to check that $\mathcal{N}_{\mathcal{T}}(R)$ is an ideal of $R$.

\begin{lemma}\label{nilpotent}
 Let $\mathcal{C}\sub\N_0^n$ be a saturated monoid. Let $R=\Z[\mathcal{C}]/I$ be a homomorphic image of the ring $\Z[\mathcal{C}]$, where $I$ is an ideal in $\Z[\mathcal{C}]$. If $R$ is additively almost-divisible, then for every $\alpha\in\mathcal{C}$, $\alpha\neq 0$, it holds that $[\x^{\alpha}]_I\in \mathcal{N}_{\mathcal{T}}(R)$.
\end{lemma}
\begin{proof}
 Let $\mathbf{D}(\mathcal{C})$ be the canonical decomposition  of $\mathcal{C}$. Every non-zero element $\alpha\in\mathcal{C}$ belongs into precisely one monoid $\mathcal{D}\in\mathbf{D}(\mathcal{C})$. We prove our assertion by the downward induction on the dimension of monoids $\mathcal{D}$  in $\mathbf{D}(\mathcal{C})$, i.e., from the highest dimension $n_0$ appearing  in $\mathbf{D}(\mathcal{C})$ to the lowest one.
 
 Let  $\mathcal{D}\in\mathbf{D}(\mathcal{C})$ have the dimension $n_0$. By Proposition \ref{monoid_decomp}\ref{i} and \ref{ii},  the monoid $\widetilde{\mathcal{D}}$ is finitely generated and $\mathcal{C}\sub\widetilde{\mathcal{D}}$. It follows that for the ideal $J=\Z[\widetilde{\mathcal{D}}]\cdot I$ of $\Z[\widetilde{\mathcal{D}}]$ the ring $R'=
 \Z[\widetilde{\mathcal{D}}]/J$ is finitely generated. Consider the natural ring homomorphism $R\to R'$, such that $[f]_{I}\mapsto [f]_{J}$ for $f\in\Z[\mathcal{C}]$.  Since the unity $1_R=[1]_{I}$ is an   additively almost-divisible element in $R$, the corresponding unity $1_{R'}=[1]_J\in R'$ is also additively almost-divisible in $R'$. Thus, by Theorem \ref{divisible}, the element $1_{R'}$ must be torsion in $R'$ and, therefore, there are polynomials $f_1,\dots, f_{m}\in \Z[\widetilde{\mathcal{D}}]$ and $h_1,\dots, h_{m}\in I$  such that $q=\sum_{i=1}^{m}f_i h_i$ for some $q\in\N$.

 Now, pick $\alpha\in\mathcal{D}\setminus\{0\}$.
 By Lemma  \ref{decomp_lemma}, there is $k\in\N$ such that $\x^{k\alpha}f_i\in \Z[\mathcal{C}]$ for every $i=1,\dots,m$. Therefore $$q\x^{k\alpha}=\sum_{i=1}^{m}(\x^{k\alpha}f_i)h_i\in \Z[\mathcal{C}]\cdot I=I$$ and $[q\x^{k\alpha}]_{I}=0$ in $R$. In particular, $[\x^{\alpha}]_{I}\in\mathcal{N}_{\mathcal{T}}(R)$.
 
 To proceed by induction, assume that for a given $n_1\in\N$,  $n_1<n_0$ and for every monoid $\mathcal{E}\in\mathbf{D}(\mathcal{C})$ such that $n_1<\dim(\mathcal{E})$ and every  $\delta\in \mathcal{C}$, $\delta\neq 0$, is $[\x^{\delta}]_{I}\in\mathcal{N}_{\mathcal{T}}(R)$.
  
  Now, let $\mathcal{D}\in\mathbf{D}(\mathcal{C})$ be a monoid such that $\dim(\mathcal{D})=n_1$. We are going to show that  $[\x^{\alpha}]_{I}\in\mathcal{N}_{\mathcal{T}}(R)$ for every $\alpha\in\mathcal{D}$, $\alpha\neq 0$.

 First, consider the ring $S=\Z[\mathcal{C}+\widetilde{\mathcal{D}}]/J$, where  $J=\Z[\mathcal{C}+\widetilde{\mathcal{D}}]\cdot I$ is an ideal in $\Z[\mathcal{C}+\widetilde{\mathcal{D}}]$. Let $\pi:S\to S/\mathcal{N}_{\mathcal{T}}(S)$ be the natural ring epimorphism and let $T$ be the subring of $S/\mathcal{N}_{\mathcal{T}}(S)$  generated by the set  $\set{\pi\big([\x^{\beta}]_{J}\big)}{\beta\in\widetilde{\mathcal{D}}}$. 

 Now, let $\alpha\in \mathcal{D}\setminus\{0\}$. The element $[1]_{I}$ is additively almost-divisible in $R$ and therefore there is an infinite set $P$ of prime numbers and $r\in\N$ such that for every $m\in P$ there is $t_m\in \Z[\mathcal{C}]$ and it holds that $[r]_{I}=m\cdot [t_m]_{I}$. Hence $[r\x^{\alpha}]_{I}=m\cdot [\x^{\alpha}t_m]_{I}$ and, by using the natural ring homomorphism $R\to S$ with $[g]_{I}\mapsto[g]_{J}$ for $g\in\Z[\mathcal{C}]$, it follows that $[r\x^{\alpha}]_{J}=m\cdot [\x^{\alpha}t_m]_{J}$.
 
  Further, by the induction hypothesis and by Lemma \ref{nil-radical}\ref{b2}, we obtain that $\pi\big([\x^{\alpha}t_m]_{J}\big)\in T$ and the set of equalities $\pi\big([r\x^{\alpha}]_{J}\big)=m\cdot \pi\big([\x^{\alpha}t_m]_{J}\big)$, where $m\in P$, implies that the element $\pi\big([r\x^{\alpha}]_{J}\big)\in T$ is additively almost-divisible in $T$.

  Finally, by Lemma \ref{nil-radical}\ref{b1}, the ring $T$ is finitely generated and thus, by Proposition \ref{divisible},  it follows that $\pi\big([r\x^{\alpha}]_{J}\big)$ is torsion in $T\sub S/\mathcal{N}_{\mathcal{T}}(S)$. Hence, there is $s\in\N$ such that $0=s\cdot \pi\big([ r\x^{\alpha}]_{J}\big)=\pi\big([ sr\x^{\alpha}]_{J}\big)$ in $S/\mathcal{N}_{\mathcal{T}}(S)$, and, therefore we obtain that  $[sr\x^{\alpha}]_{J}\in \mathcal{N}_{\mathcal{T}}(S)$. By the definition of $\mathcal{N}_{\mathcal{T}}(S)$, there are  $\ell,q'\in\N$ such that  $[ q'(sr\x^{\alpha})^{\ell}]_{J}=0$ in $S$.   Hence there are polynomials $e_1,\dots, e_{p}\in \Z[\mathcal{C}+\widetilde{\mathcal{D}}]$ and polynomials $g_1,\dots, g_{p}\in I$ such that $q''\x^{\ell\alpha}=\sum_{i=1}^pe_i g_i$, where $q''=q'(sr)^{\ell}\in\N$.
By Lemma  \ref{decomp_lemma}, we obtain that there is $k\in\N$ such that $\x^{k\alpha}e_i\in \Z[\mathcal{C}]$ for every $i=1,\dots,p$. Therefore $$q''\x^{(k+\ell)\alpha}=\sum_{i=1}^p(\x^{k\alpha}e_i)g_i\in \Z[\mathcal{C}]\cdot I=I$$ and $[q''\x^{(k+\ell)\alpha}]_I=0$ in $R$. In particular, we have obtained that $[\x^{\alpha}]_I\in \mathcal{N}_{\mathcal{T}}(R)$.

This concludes our proof and, indeed, for every $0\neq\alpha\in\mathcal{C}$ we have obtained that $[\x^{\alpha}]_I\in \mathcal{N}_{\mathcal{T}}(R)$.
\end{proof}

\begin{remark}\label{integral}
Let us recall a basic property of integral extensions of rings. Let $A\leq B$ be an integral extension of commutative rings with the same unity $1_A=1_B=1$. If $I$ is an ideal of $A$ and  $1\in B\cdot I$ then $1\in I$.  
\end{remark}

\begin{theorem}\label{main_theorem_rings}
 Let $R$ be  a ring that is a factor of the monoid ring $\Z[\mathcal{C}]$, where $\mathcal{C}$ is a submonoid of $\N^{n}_{0}(+)$. If $R$ is additively almost-divisible then the ring $R$ is torsion.
\end{theorem}
\begin{proof}
Let $R$ be an additively almost-divisible ring  isomorphic to $\Z[\mathcal{C}]/I$ for some ideal $I$ in $\Z[\mathcal{C}]$.

Assume first, that the monoid $\mathcal{C}$ is saturated. According to Lemma \ref{nilpotent} the ring $S=R/\mathcal{N}_{\mathcal{T}}(R)$ is generated by a single element - the unity $1_S$. Also, $S$ is an additively almost-divisible ring. Hence, by Theorem \ref{divisible}, the ring $S$ is torsion. Thus $1_R\in\mathcal{N}_{\mathcal{T}}(R)$ and the ring $R$ is torsion. 

Now, let $\mathcal{C}$ be a general submonoid of $\N^{n}_{0}(+)$. Put $$\mathcal{C}'=\set{\alpha\in\N^{n}_{0}}{(\exists\ k\in\N)\ k\alpha\in\mathcal{C}}.$$
Then $\mathcal{C}'$ is a saturated submonoid of $\N^{n}_{0}(+)$. Set $I'=\Z[\mathcal{C}']\cdot I$. Then the ring $R'=\Z[\mathcal{C}']/I'$ is additively almost-divisible. By the first part of the proof, we obtain that $R'$ is torsion and, therefore, there is $m\in\N$  such that $m\in I'=\Z[\mathcal{C}']\cdot I$.

Further, $\Q[\mathcal{C}]\leq \Q[\mathcal{C}']$ is an integral extension of commutative rings (with the same unity) and the set $J=\Q\cdot I$ is an ideal of $\Q[\mathcal{C}]$. Then $\Q[\mathcal{C}']\cdot J=\Q\cdot \Z[\mathcal{C}']\cdot I=\Q\cdot I'$. From the fact that $m\in I'$ it follows that $1\in \Q[\mathcal{C}']\cdot J$. Thus, by Remark \ref{integral}, we obtain that $1\in J=\Q\cdot I$. In particular, this means that there is $k\in\N$ such that $k\in I$.

Therefore the ring $R$ is torsion, indeed. 
\end{proof}

In the next proposition we recall the structure of subrings of the field $\Q$ to show how the additive almost-divisibility is applied in this case.

\begin{proposition}\label{subrings}
 Let $R$ be a subring of $\Q$ such that $R\not\sub\Z$. Then there is $n\in\N$ and a non-empty set of prime numbers $P\sub\N$ such that $\gcd(n,p)=1$ for every $p\in P$ and $R=\set{\frac{nk}{q}}{q\in\N\ \text{is a product of primes from}\ P,\ k\in\Z}$. 
 
 The ring $R$ is finitely generated if and only if $P$ is a finite set. Moreover, $R$ is not finitely generated if and only if $R$ is additively almost-divisible.
\end{proposition}
\begin{proof}
 Since $R\neq \{0\}$, there is $n\in\N$ such that $R\cap \Z=n\Z$. Let
 $P$ be the set of all prime numbers $p\in\N$ that there are $a,b\in\N$ such that $\tfrac{a}{pb}\in R$ and $\gcd(a,pb)=1$. As $R\not\sub\Z$, clearly, $P\neq\emptyset$. 
 
 Further, for $p\in P$, by the definition of $P$, we have that $\frac{a}{p}=b\frac{a}{pb}\in R$ for some $a,b\in\N$ such that $\gcd(a,p)=1$. Further, $a=p\frac{a}{p}\in\R\cap\Z$ and therefore $a=nd$ for some $d\in\N$ and $\gcd(n,p)=\gcd(d,p)=1$. Hence there are $\alpha,\beta\in\Z$ such that $\alpha d+\beta p=1$ and we obtain that $R$ contains the element $\alpha\frac{nd}{p}+\beta n=\frac{n(\alpha d+\beta p)}{p}=\frac{n}{p}$. 
 
 Set $\overline{P}=\set{q\in\N}{q\ \text{is a product of primes from}\ P}$.  Now, we show that $\frac{n}{q}\in R$ for every $q\in\overline{P}$. Let $q=p_1\cdots p_k$ where $p_1,\dots,p_k\in P$. According to the previous part of the proof $\frac{n^k}{q}=\frac{n}{p_1}\cdots\frac{n}{p_k}\in R$. As $\gcd(n,q)=1$, there are $\gamma,\delta\in\Z$ such that $\gamma n^{k-1}+\delta q=1$ and therefore we obtain that $\frac{n}{q}=\frac{n(\gamma n^{k-1}+\delta q)}{q}=\gamma \frac{n^{k}}{q}+\delta n\in R$. In this way we have also shown that the set $\{n\}\cup\set{\frac{n}{p}}{p\in P}$ generates the ring $R$ and that, indeed, it holds that  $R=\set{\frac{nk}{q}}{k\in\Z, q\in \overline{P}}$. Obviously, if $R$ is finitely generated, then only finitely many prime numbers appear  in the denominators of elements of $R$ and, therefore, $P$ is finite. 
 
The ring $R$ is now clearly additively almost-divisible if and only if $P$ is infinite and, consequently, if and only if $R$ is not finitely generated. 
\end{proof}

\begin{corollary}\label{corollary_rings}
  Let $R$ be  a ring that is a factor of the monoid ring $\Z[\mathcal{C}]$, where $\mathcal{C}$ is a submonoid of $\N^{n}_{0}(+)$. Then $R$ cannot contain as its subring any non-finitely generated subring of $\Q$.
\end{corollary}
\begin{proof}
Let $R$ be a factor of the monoid ring $\Z[\mathcal{C}]$ for some submonoid $\mathcal{C}$ is of $\N^{n}_{0}(+)$. Assume for contrary that $R$ contains as its subring a non-finitely generated subring $T$ of $\Q$. By Proposition \ref{subrings}, there is $n\in\N$ and an infinite set of prime numbers $P$ such that $\gcd(n,p)=1$ for every $p\in P$ and $T=\set{\frac{nk}{q}}{k\in\Z,\ q\in \overline{P}}$. Consider now the quotient ring $S=R[x]_{/(1-nx)}$ of the polynomial ring  $R[x]$ by its ideal generated by the polynomial $1-nx$. 

Let us show that the natural map $\varphi: T\to R[x]_{/(1-nx)}$ is an embedding. Let $a\in T$ and $f(x)=a_{k}x^{k}+\cdots+a_1 x+a_0\in R[x]$, for $k\geq 0$ be such that $a=f(x)\cdot (1-nx)=-na_kx^{k+1}+(a_{k}-na_{k-1})x^k+\cdots+(a_1-na_0)x+a_0$. Then $a_0=a$, $na_k=0$ and $a_{i+1}=na_{i}$ for every $i=0,\dots,k-1$. It follows that $n^{k+1}a=0$ and $a=0$. Hence $\varphi$ is indeed an embedding.

Let $b\in S$ be the inverse of the element $n\cdot 1_S$. Now, obviously, the set $T'=\varphi(T)\cdot b$ is a unitary subring of $S$ and it is isomorphic to a non-finitely generated subring of $\Q$. Further, $S$ is a factor of the monoid ring  $\Z[\mathcal{C}']$, where $\mathcal{C}'=\mathcal{C}\oplus\N_{0}$.

Finally, let $1_{T'}=n\cdot b$ be the unity of $T'$. Consider the ring epimorphism $\psi: S\rightarrow S\cdot 1_{T'}$. The ring $S'=S\cdot 1_{T'}$ is now a factor of the monoid ring  $\Z[\mathcal{C}']$, where $\mathcal{C}'=\mathcal{C}\oplus\N_{0}\sub\N_{0}^{k+1}$, and $S\cdot 1_{T'}$ contains a unitary subring $T'$ that is isomorphic to a non-finitely generated subring of $\Q$. Since $S'$ shares the unity with $1_T'$ and $T'$ is additively almost-divisible,  the ring $S'$ is additively almost-divisible, by Proposition \ref{subrings}, too. By Theorem \ref{main_theorem_rings}, $S'$ is torsion. This is a contradiction with the fact that $1_T'$ is torsionless.

We have shown that $R$ cannot contain a non-finitely generated subring of $\Q$.
\end{proof}

\section{Additively almost-divisible factors of monoid semiring $\N[\mathcal{C}]$  are  torsion}\label{section 3}

In this section we prove our main result on semirings. First, let us recall the following helpful notion. There is a well-know construction of the Grothendieck group that produces an abelian group from an additive commutative semigroup. If we adopt the same construction for a commutative semiring $S$ we obtain a commutative ring $G(S)$. We only need to define the multiplicative operation. 

In particular, the construction is as follows. On the set $S\times S$ we define the following operations $\oplus$ and $\odot$ as $$(x,y)\oplus(x',y')=(x+x',y+y')\ \ \text{and}\ \ (x,y)\odot(x',y')=(xx'+yy',xy'+x'y)$$ and a relation $\approx$ as $$(x,y)\approx (x',y')\ \Leftrightarrow\ (\exists t\in S)\ x+y'+t=x'+y+t$$
for every $x,x',y,y'\in S$. With the same effort as for the Grothendieck group one can check that $(S\times S)(\oplus,\odot)$ is a semiring and $\approx$ is a congruence on this semiring. Now we simply set $G(S)=(S\times S)_{/\approx}$. The map $\sigma_S:S\to G(S)$, defined as $\sigma_S(x)=(2x,x)_{/\approx}$ for $x\in S$,  is then a semiring homomorphism (similarly as for the Grothendieck semigroup). Let us notice that $(x,y)_{/\approx}=\sigma_{S}(x)-\sigma_{S}(y)$ for all $x,y\in S$. 

\

Unfortunately, some properties of the ring $G(S)$ do not imply easily  the corresponding properties of the original semiring $S$. Therefore we will need another auxiliary notion.  For a semiring $S$ with a unity $1_S$ we put $$Q_{S}=\{x\in S|(\exists k\in \N)(\exists a\in S)\ x+a= k\cdot 1_{S}\}\ .$$  

\begin{proposition}\label{q}
Let $S$ be semiring with a unity. Then $Q_{S}$ is a subsemiring of $S$ and for all $x,y\in S$ such that $x+y\in Q_{S}$ is $x,y\in Q_{S}$. Moreover, $S$ is additively almost-divisible if and only if $Q_S$ is additively almost-divisible. 
\end{proposition}
\begin{proof}
 Clearly, $1_S\in Q_S$. Let $u,v\in Q_S$. Then $u+a= k\cdot 1_{S}$ and $v+b= \ell\cdot 1_{S}$ for some $a,b\in S$ and $k,\ell\in\N$. Hence $k\ell\cdot 1_{S}=(u+a)(v+b)=uv+(ub+av+ab)$ and $(k+\ell)\cdot 1_{S}=(u+v)+(a+b)$ and therefore $Q_S$ is a subsemiring. Clearly, if $x,y\in S$ are such that $x+y\in\Q_S$ then both $x$ and $y$ belong to $Q_S$.
 
 Assume that $S$ is additively almost-divisible. Then the element $1_S$ is almost divisible (in $S$). Hence there are $k\in\N$ and an infinite set $P$ of prime numbers such that for every $p\in P$ there is $a\in S$ such that $k\cdot 1_{S}=p\cdot a=a+(p-1)\cdot a$. Therefore $a\in Q_S$ and using the same equality we obtain that $1_S$ in an almost divisible element in $Q_S$. Thus the semiring $Q_S$ is additively almost-divisible. The opposite implication is obvious.   
\end{proof}

\begin{proposition}\label{equivalence_divisibility}
The semiring $S$ is torsion if and only if $S$ is additively almost-divisible and the  ring  $G(Q_{S})$ is torsion. 	
\end{proposition}
\begin{proof}
Let  $S$ be additively almost-divisible and let $G(Q_{S})$ be torsion. Then there are $t\in Q_{S}$ and $k\in\N$, such that $t=k\cdot 1_{S}+t$. By the definition of $Q_S$, there are $k'\in\N$ and $s\in S$ such that $t+s=k'\cdot 1_{S}$. Therefore we obtain $k'\cdot 1_{S}=k\cdot 1_{S}+k'\cdot 1_{S}$ and the element $1_S$ is torsion in $S$. Hence the semiring $S$ is torsion.  

The other implication follows immediately from the diagram in Section \ref{preliminaries}.
\end{proof}

Let $P$ be a non-empty subset of the set of all prime numbers. Consider the following subsemiring of $\Q^+$: $$\N_{P}=\set{\tfrac{k}{n}}{k\in\N,\ n \ \text{is a product of primes from}\ P }\ .$$

Clearly, $\N_P$ is the positive part of a corresponding localization of the ring $\Z$.

\

The following proposition provides a simple characterization of strong almost-divisibility for semirings with a unity.

\begin{proposition}
 Let $S$ be a semiring with a unity $1_S$. Then the following are equivalent:
 \begin{enumerate}
  \item[(i)] There is an infinite set of primes $P$ and a semiring homomorphism $\varphi:\N_P\rightarrow S$ such that $\varphi(1_{\N_P})=1_S$.
  \item[(ii)] $S(+)$ is strongly almost-divisible.
 \end{enumerate}
\end{proposition}
\begin{proof}
 (i) $\Rightarrow$ (ii) The unity $1_{\N_P}$ is strongly almost-divisible. So is the unity $1_S=\varphi(1_{\N_P})$ and $S(+)$ is therefore strongly almost-divisible.
 
 (ii) $\Rightarrow$ (i) Let $P$ be an infinite set of primes such that for every $p\in P$ there is $a_p\in S$ such that $1_S=(p\cdot 1_S) \cdot a_p$. The element $a_p$ is invertible and therefore it is uniquely determined. For $k\in\N$ and $p_1,\dots,p_{m}\in P$ put $\varphi\left(\frac{k}{p_1\cdots p_{m}}\right)=k\cdot a_1\cdots a_m$. It is now easy to verify that $\varphi: \N_P\to S$ is a well defined semiring homomorphism.
 \end{proof}

 We are now in position to prove our main result. Let us recall that $\N[\mathcal{C}]=\set{\sum_{i=1}^m k_i \x^{\alpha_i}}{m\in\N, k_i\in\N, \alpha_i\in\mathcal{C}}$ is a subsemiring of the ring $\Z[\mathcal{C}]$ for  a submonoid $\mathcal{C}$ of $\N^{n}_{0}(+)$ for some $n\in\N$.

\begin{theorem}\label{main-theorem}
Let $S$ be a factor of a monoid semiring $\N[\mathcal{C}]$, where $\mathcal{C}$ is a submonoid of $\N^{n}_{0}(+)$. Then:
\begin{enumerate}
\item[(i)] The semigroup $S(+)$ is almost-divisible if and only if $S(+)$ is torsion.
\item[(ii)] The semigroup $S(+)$ is strongly almost-divisible  if and only if $S(+)$ is both regular and torsion.
\item[(iii)] The semiring $S$ cannot contain the subsemiring $\N_P$ for any infinite set $P$ of prime numbers.
\end{enumerate}
\end{theorem}
\begin{proof}
(i) By Proposition \ref{equivalence_divisibility}, we only need to show that the ring $G(Q_S)$ is torsion. 

Let $\pi:\N[\mathcal{C}]\to S$ be a semiring epimorphism. Put $\mathcal{B}=\set{\alpha\in\mathcal{C}}{\pi(\x^{\alpha})\in Q_S}.$
Clearly, $\mathcal{B}$ is a submonoid of $\mathcal{C}$ and the set $\set{\pi(\x^{\alpha})}{\alpha\in\mathcal{B}}$ generates the additive semigroup $Q_S(+)$. The additive group $G(Q_S)(+)$ of the ring $G(Q_S)$ is therefore generated by the set $\set{\sigma(\pi(\x^{\alpha}))}{\alpha\in\mathcal{B}}$, where $\sigma=\sigma_{Q_S}:Q_S\to G(Q_S)$ is the corresponding semiring homomorphism. Thus the ring $G(Q_S)$ is a homomorphic image of the ring $\Z[\mathcal{B}]$. If $S$ is additively almost-divisible then  $Q_S$ is additively almost-divisible, by \ref{q}. Also the ring $G(Q_S)$ is then additively almost-divisible.
Now, by Proposition \ref{main_theorem_rings}, the ring $G(Q_S)$ is trivial. Therefore, by Proposition \ref{equivalence_divisibility}, the semiring $S$ is additively idempotent. 

The opposite implication is obvious.

(ii)  By (i), we only need to prove that if $S(+)$ is strongly almost-divisibility then $S(+)$ is regular. The semiring $S$ has a unity $1_S$. By (i), $1_S$ is torsion and therefore there is $k\in\N$ such that $k\cdot 1_S=2k\cdot 1_S$. Consequently,  $kx=2kx$ for every $x\in S$. By the basic properties of the cyclic semigroups, the set $G_x=\set{\ell x}{\ell\in\N, \ell\geq k}$ is a finite additive group for every $x\in S$. Now, by the strong almost-divisibility, for every $a\in S$ there is a prime number $p\in\N$ big enough such that $a=px\in G_x$ for some $x\in S$. Thus the element $a$ is contained in an additive group and the semigroup $S(+)$ is therefore regular. 

(iii) Assume for contrary, that there is an infinite set $P$ of primes such that $\N_P$ is a subsemiring of $S$. Let  $S'=S\cdot 1_{\N_P}$. Then $S'$ is a subsemiring of $S$ and  the map $\psi:S\to S'$, where $\psi(a)=a\cdot 1_{\N_P}$, is a semiring epimorphism. Now, $1_{S'}=1_{\N_P}$ and the additive semigroup $S'(+)$ is therefore strongly almost-divisible. By (ii), $S'(+)$ is torsion, while $1_{S'}=1_{\N_P}$ is a torsionless element in $\N_P$, a contradiction.
\end{proof}

We have therefore confirmed Conjecture \ref{conjecture} for commutative semirings with a unity.


\end{document}